\documentclass[a4paper,reqno,10pt]{amsart}
\usepackage{amsmath,amstext,amssymb,amsopn,amsthm,mathrsfs}
\usepackage{xcolor}
\usepackage[matrix,arrow,ps]{xy}

\numberwithin{equation}{section}
\allowdisplaybreaks

\newtheorem{theorem}{Theorem}[section]
\newtheorem{proposition}[theorem]{Proposition}
\newtheorem{lemma}[theorem]{Lemma}
\newtheorem{corollary}[theorem]{Corollary}

\def\a{\alpha}
\def\b{\beta}

\def\vp{\varphi}

\newcommand{\N}{\mathbb{N}}

\newcommand{\R}{\mathbb{R}}
\newcommand{\C}{\mathbb{C}}
\newcommand{\Rd}{\mathbb{R}^d}

\newcommand{\bbG}{\mathbb{G}}
\newcommand{\bbH}{\mathbb{H}}
\newcommand{\bbM}{\mathbb{M}}

\newcommand*\calF{\mathcal{F}}
\newcommand*\calG{\mathcal{G}}
\newcommand*\calH{\mathcal{H}}

\newcommand{\D}{{\rm Dom}}

\newcommand{\calka}{\int_{\R^{d'}}}
\begin{document}
\title[A transform for the Grushin operator]{A transform for the Grushin operator with applications}

\author[K. Stempak]{Krzysztof Stempak}
\address{Krzysztof Stempak \endgraf\vskip -0.1cm
         55-093 Kie\l{}cz\'ow, Poland \endgraf \vskip -0.1cm
				        }
\email{Krzysztof.Stempak@pwr.edu.pl}

\begin{abstract} 
In the setting of the Grushin differential operator $G=-\Delta_{x'}-|x'|^2\Delta_{x''}$ with domain 
${\rm Dom}\,G=C^\infty_c(\mathbb{R}^d)\subset L^2(\mathbb{R}^d)$, 
we define a scalar transform which is a mixture of the partial Fourier transform and a transform based on the scaled Hermite functions. 
This transform unitarily intertwines $G$ with a multiplication operator by a nonnegative real-valued function on an appropriately associated 
`dual' space $L^2(\Gamma)$. This allows to construct a self-adjoint extension $\mathbb G$ of $G$ as a simple realization of this multiplication 
operator. Another self-adjoint extensions of $G$ are defined in terms of  sesquilinear forms and then these extensions are compared. Aditionally, 
a closed formula for the heat kernel that corresponds to the heat semigroup $\{\exp(-t\mathbb G)\}_{t>0}$ is established.
\end{abstract}




\maketitle

\section{Introduction} \label{sec:intro} 

The \textit{Grushin differential operator} $G$ on $\R^d=\R^{d'}\times\R^{d''}$, $d', d''\ge1$, is given by the differential expression
$$
-\Delta_{x'}-|x'|^2\Delta_{x''},
$$ 
where $x'\in\R^{d'}$, $x''\in\R^{d''}$,  and $\Delta_{x'},\Delta_{x''}$ are the Laplacians on $\R^{d'}$ and $\R^{d''}$, respectively. 
It is an important example of second order subelliptic differential operator with polynomial coefficients, which was intensively investigated 
in the PDE's theory and also from the point of view of harmonic analysis. There is an interesting interplay between $G$ and differential operators 
on some Lie groups, notably between $G$ and the sublaplacians on the Heisenberg groups or, more generally, on stratified Lie groups, see Dziuba\'nski 
and Sikora \cite{DzS}, and Martini and Sikora \cite{MS}. 

It is easily seen that initially considered with domain $\D\,G=C^\infty_c(\R^d)\subset L^2(\Rd)$, $G$ is symmetric and nonnegative. Moreover, 
$G$ is essentially self-adjoint (see \cite{MS} for a group-theoretic argument or \cite{DaM} for a simple proof of this), hence it admits the 
unique self-adjoint extension on $L^2(\Rd)$ called the \textit{Grushin operator} and denoted $\bbG$. During the last years numerous papers were 
devoted to different aspects of harmonic analysis of $\bbG$. See, for instance, Dall'Ara and Martini \cite{DaM} with extensive literature 
therein, and Jotsaroop et al. \cite{JST}, where, among others, the spectral decomposition of the Grushin operator was proposed (in the case $d''=1$). 

In this paper we construct and study a transform, denoted $\calG$, with integral kernels being tensor products of eigenfunctions of scaled Hermite operators 
$-\Delta_{x'}+|\xi''|^2|x'|^2$ on $L^2(\R^{d'})$, $0\neq\xi''\in\R^{d''}$, and eigenfunctions of the Laplacian $-\Delta_{x''}$ on $L^2(\R^{d''})$. 
These tensor products are eigenfunctions of $G$. The eigenfunctions of scaled Hermite operators consist of systems of scaled Hermite functions that 
form orthonormal bases in $L^2(\R^{d'})$. The use of these systems  was initiated and successfully applied by Thangavelu in the harmonic analysis 
of the sub-Laplacian on the Heisenberg group, \cite{T1}, \cite{T2}, and in an investigation of an analogue of Hardy's theorem for the Heisenberg 
group, \cite{T3}, \cite{T4}, and also in some other contexts. The idea of this transform was implicit in many works devoted to the Grushin operator, 
see e.g. \cite{MS}, but did not receive an explicit description.

The transform $\calG$ plays the same role for $G$ as the Fourier transform plays for the Laplacian. In Theorem~\ref{thm:first} we prove that $\calG$ 
is a unitary bijection between $L^2(\Rd)$ and an appropriately associated `dual' space $L^2(\Gamma)$. Moreover, $\calG$ unitarily intertwines $G$ 
with the self-adjoint multiplication operator $\bbM_\Theta$ on $L^2(\Gamma)$, where $\Theta$ is a real-valued positive function on $\Gamma$, 
see \eqref{3.2} and Section~\ref{sec:self}. This allows to construct 
a simple self-adjoint realization of $G$, denoted $\bbG$, just by mapping $\bbM_\Theta$ onto $L^2(\Rd)$; see Proposition~\ref{pro:zero}. 
Then  $\bbG$ is recognized as an operator with domain being a Sobolev-type space adapted to $G$ and the action  $\bbG f=Gf$, where $Gf$ 
is understood in the distributional sense, see Proposition~\ref{pro:Sob2}. This is much the same as in the case of the Laplacian on $L^2(\Rd)$.

Consequently, the functional calculus for $\bbG$ reflected from that for $\bbM_\Theta$ allows to establish a closed 
formula of the heat kernel $\{p_t\}_{t>0}$ corresponding to the heat semigroup $\{\exp(-t\bbG)\}_{t>0}$; this is done in Theorem~\ref{thm:heat}. 
Closed formulas of $p_t$ were known, see Section 5.2 in the monograph \cite{CCFI}, where the case $d'=d''=1$ was treated, Garofalo and Tralli 
\cite[Theorem 3.4]{GT}, or Oliveira and Viana \cite[Proposition 2.1]{OV}. It is also known, at least for $d''=1$, that the heat kernel for the 
Grushin operator is the image, under the unitary Schr\"odinger representation of the Heisenberg group $\bbH_{d'}=\R^{d'}\times\R^{d'}\times\R$, 
of the heat kernel for the sublaplacian on $\bbH_{d'}$, cf. \cite{DJ}. We mention that in \cite{GT} and \cite {OV},  
in the framework of the differential operator $G$ and the associated heat equation  $\partial_tu=-G_xu$, the heat kernel is meant as
a function $p(x,\xi,t)$ such that for any $\xi\in\Rd$ the function $p(\cdot,\xi,\cdot)$ is a solution of the heat equation, and, for every $x\in\Rd$, 
$p(x,\cdot,t)\to \delta_x$ in the distributional sense as $t\to0^+$. Our approach is operator theoretic and the heat kernel is understood as the 
collection of integral kernels of $\exp(-t\bbG)$, which are bounded operators on $L^2(\Rd)$. We add that calculations leading to \eqref{5.2}, the 
closed form of $p_t$, are straightforward and, we believe, are worth presentation. 

The structure of the paper is as follows. Section~\ref{sec:prel} contains preliminaries. In Section~\ref{sec:Gru} we define the $\calG$-transform 
and the inverse transform $\calG^{-1}$, and prove Plancherel's identities and the inversion formulas for them, which are the main tools used in the next sections. 
Section~\ref{sec:self} is devoted to establishing self-adjoint extensions of $G$ based either on the $\calG$-transform or on a sesquilinear form, 
and showing that they coincide. In Theorem~\ref{thm:heat} of Section~\ref{sec:heat} a compact formula of the heat kernel is established and basic 
properties of it are examined. Finally, in Section~\ref{sec:app} we gather proofs of some technical results used earlier; we believe that putting 
this stuff in the appendix will allow the reader to concentrate on the main line of thoughts. 

\textbf{Notation and terminology}. 
In general the symbols $'$ and $''$ (prime and double prime) used as superscripts will indicate that the related objects correspond to the $\R^{d'}$ 
or $\R^{d''}$ settings, respectively. The symbol $|\cdot|$ will stand for the Euclidean norm  in  $\R^{d'}$ or $\R^{d''}$, depending on the context,  
and  $\langle\cdot ,\cdot\rangle$ will denote the usual inner product in $L^2(\R^{d})$. 
We use the calligraphic letters $\calF$, $\calG$, $\calH$, possibly with subscripts and/or superscripts, to denote transforms. For instance $\calF''$ 
stands for the Fourier transform on $\R^{d''}$,  
$$
\calF'' f(\xi'')=(2\pi)^{-d''/2}\int_{\R^{d''}}f(x'')e^{-\textrm{i}\xi''\cdot x''}\,dx'', \qquad f\in L^1(\R^{d''}),
$$
and $\calF''$ also means the usual extension of this transform onto $L^2(\R^{d''})$ being an isometric bijection. Then $\calF''^{-1}$ stands for the 
inverse Fourier transform. If not otherwise stated $\calF''$ will always be understood as the unitary automorphism of $L^2(\R^{d''})$.
Given a function $f$ 
on $\R^{d'}\times\R^{d''}$ we shall write $f_{x'}:=f(x',\cdot)$ and $f_{x''}:=f(\cdot,x'')$ to denote the sections of $f$ with respect to the `first' 
variable $x'\in \R^{d'}$, or the `second' variable $x''\in \R^{d''}$. Throughout, the decomposition $\Rd=\R^{d'}\times\R^{d''}$ is fixed and $x'$ and 
$x''$ will always be taken from the associated decomposition $x=(x',x'')$. 
\section{Preliminaries} \label{sec:prel}  
The system of multidimensional Hermite functions $\{h_k\}_{k\in \N^{d'}}$ is an orthonormal basis in $L^2(\R^{d'})$ consisting of eigenfunctions of 
the Hermite differential operator, 
$$
(-\Delta_{x'}+|x'|^2)h_k=(2|k|+d')h_k, \qquad |k|:=k_1+\ldots+k_{d'};
$$
see \cite{T0} or \cite{ST} for details (exceptionally, $|\cdot|$ stands here also for the \textit{length}  of $k$). The \textit{scaled systems} 
$\{h_{k,\tau}\}_{k\in \N^{d'}}$,  $\tau>0$,
$$
h_{k,\tau}(x')=\tau^{d'/4}h_k(\tau^{1/2}x'), \qquad x'\in \R^{d'},
$$
form orthonormal bases in $L^2(\R^{d'})$ and consist of eigenfunctions of the \textit{scaled Hermite differential operator}, 
$$
(-\Delta_{x'}+\tau^2|x'|^2)h_{k,\tau}=\lambda_k'\tau h_{k,\tau},\qquad \lambda_k'=2|k|+d'.
$$
Now define the \textit{scaled Hermite transform} which for every $\xi''\in\R^{d''}_0:=\R^{d''}\setminus\{0\}$ attaches to a function $f\in L^2(\R^{d'})$ 
the sequence of coefficients from the expansion of $f$ with respect to $\{h_{k,|\xi''|}\}_{k\in \N^{d'}}$. Namely, for $f\in L^2(\R^{d'})$, 
$$
\calH^{\textsl{sc}}_{\xi''} f(k):=\calH^{\textsl{sc}} f(k,\xi'') =\langle f,h_{k,|\xi''|}  \rangle, \qquad k\in\N^{d'},\,\,\,\xi''\in\R^{d''}_0.
$$
By Parseval's identity, for every $\xi''\in\R^{d''}_0$, $\calH^{\textsl{sc}}_{\xi''}\colon L^2(\R^{d'})\to \ell^2(\N^{d'})$ is a unitary isomorphism, 
and its inverse is the mapping 
$$
\calH^{\textsl{sc};-1}_{\xi''}\colon \ell^2(\N^{d'}) \to  L^2(\R^{d'}),\qquad \{c_k\}_{k\in \N^{d'}}\to\sum_{k\in \N^{d'}}c_k h_{k,|\xi''|}.
$$

The appropriate scaling of the Hermite functions is crucial in verification that the functions
$$  
\psi_{k,\xi''}(x):=(2\pi)^{-d''/2}h_{k,|\xi''|}(x')e^{-\textrm{i} x''\cdot\xi''}, \qquad k\in\N^{d'},\quad \xi''\in\R^{d''}_0,
$$  
are eigenfunctions of the Grushin differential operator. (The normalization constant $(2\pi)^{-d''/2}$ is, of course, immaterial in this place, 
but it is decisive later on to keep some transforms $L^2$ isometric.)
\begin{lemma} \label{lem:second}
For $k\in\N^{d'}$ and $\xi''\in\R^{d''}_0$, 
$$
G\psi_{k,\xi''} = \lambda_{k}'|\xi''| \, \psi_{k,\xi''}.
$$
\end{lemma}
\begin{proof} 
The following elementary checking is included for completness. We have
\begin{align*}
(2\pi)^{d''/2}G\,\psi_{k,\xi''}(x)
&=(-\Delta_{x'}-|x'|^2\Delta_{x''})\big( h_{k,|\xi''|}(x')e^{-\textrm{i}x''\cdot\xi''}\big) \\
&=-\big(\Delta_{x'}h_{k,|\xi''|}\big)(x')e^{-\textrm{i}x''\cdot\xi''} - |x'|^2 h_{k,|\xi''|}(x')\Delta_{x''}e^{-\textrm{i}x''\cdot\xi''}\\
&=(-\Delta_{x'}+|\xi''|^2|x'|^2)h_{k,|\xi''|}(x') e^{-\textrm{i}x''\cdot\xi''}  \\ 
&=(2\pi)^{d''/2}\lambda_{k}'|\xi''|\, \psi_{k,\xi''}(x).
\end{align*}
\end{proof}

We shall need $L^1$ and $L^\infty$ estimates of the Hermite functions $h_k$ which follow from pointwise uniform estimates proved in the one-dimensional 
setting by Askey and Wainger, and enhanced by Muckenhoupt \cite[(2.3)]{Mu}; see also \cite[Section 1.5]{T0} or \cite[Section 2]{ST} for additional comments. Namely, with $|k|\to\infty$,
\begin{equation} \label{fiu}
\|h_k\|_{L^1(\R^{d'})}=O(|k|^{d'/4}), \qquad \|h_k\|_{L^\infty(\R^{d'})}=O(1).
\end{equation}
\section{$\calG$-transform and its inverse} \label{sec:Gru} 

In this section we define a transform, in the present context playing the role of the Fourier transform in the analysis of the Laplacian, which 
will be crucial for the spectral analysis of the Grushin operator. Using the scaled Hermite transform in this definition is motivated by \cite{JST}, 
where the system of scaled Hermite functions was applied to describe the spectral decomposition of the Grushin operator.

For a suitable function $f$ on $\R^d$ we define its transform $\calG f$, a function on $\N^{d'}\times\R^{d''}_0$, given as the integral transform based 
on the eigenfunctions $\psi_{k,\xi''}$  of $G$. Namely, we set
\begin{equation} \label{3.1}
\calG f(k,\xi'')=\int_{\Rd} f(x)\psi_{k,\xi''}(x)\,dx, \qquad k\in\N^{d'},\quad \xi''\in\R^{d''}_0,
\end{equation}
provided $f$ is such that the integral converges for every $k\in\N^{d'}$ and almost every $\xi\in\R^{d''}_0$. Obviously, this happens for 
$f\in C^\infty_c\big(\Rd\big)$. Notice, that using a version of Green's formula adapted to $G$, see Lemma~\ref{lem:int}, 
for $\vp\in C^\infty_c\big(\Rd)$ we obtain
\begin{equation}\label{3.2}
\calG (G \vp)(k,\xi'')=\lambda_{k}'|\xi''|\, \calG \vp(k,\xi''), \qquad k\in\N^{d'},\quad  \xi''\in\R^{d''}_0.
\end{equation}
Thus, $\calG$ intertwines the differential operator $G$ and the multiplication operator by $(k,\xi'')\to \lambda_{k}'|\xi''|$. It will follow from 
Theorem~\ref{thm:first} that with appropriately (and naturally) defined $L^2$ space, $\calG$ becomes a unitary isomorphism. 

The $\calG$-transform is for $\vp\in C^\infty_c\big(\Rd\big)$ expressible through the Fourier transform $\calF''$ and the scaled Hermite  
transform $\calH^{\textsl{sc}}$. Recall that  $\calF''$ is a unitary automorphism of $L^2(\R^{d''})$ and $\calH^{\textsl{sc}}$  
acts on $L^2(\R^{d'})$ attaching to $g\in L^2(\R^{d'})$ the sequences $\{\langle g,h_{k,|\xi''|}\rangle\}_{k\in\N^{d'}}$, $\xi''\in \R^{d''}_0$. Namely,
\begin{equation*}
\calG \vp(k,\xi'')=\int_{\R^{d'}} h_{k,|\xi''|}(x') \calF'' \vp_{x'}(\xi'')\,dx'=\calH^{\textsl{sc}}_{\xi''}\big[\calF'' \vp_{x'}(\xi'')\big](k).
\end{equation*}
The expression inside the square brackets is, for fixed $\xi''$, a function of $x'$ to which $\calH^{\textsl{sc}}_{\xi''}$ is applied. 
Similarly to the above case, in several places below, square brackets will indicate that expressions inside them should be properly identified. 
This means, that one has to recognize a variable among some other characters; from the context it will be clear which characters appear as fixed parameters.

This preliminary calculation leads to the extension of the action of $\calG$ onto $L^2(\Rd)$ by setting
\begin{equation}\label{3.3}
\calG f(k,\xi''):=\calH^{\textsl{sc}}_{\xi''}\big[\calF''f_{x'}(\xi'')\big](k), 
\qquad f\in L^2(\Rd).
\end{equation}
It remains to confirm correctness of the above definition. But
\begin{align*}
\int_{\R^{d''}}\calka | \calF'' f_{x'}(\xi'')|^2dx'\,d\xi''=\calka\int_{\R^{d''}}|\calF'' f_{x'}(\xi'')|^2d\xi''\,dx'&=
\calka\int_{\R^{d''}}|f_{x'}(\xi'')|^2d\xi''\,dx'\\
&=\|f\|^2_{L^2(\Rd)}<\infty,
\end{align*}
so that $x'\to \calF''f_{x'}(\xi'')$ is in $L^2(\R^{d'})$ for a.e. $\xi''\in \R^{d''}$, hence $\calH^{\textsl{sc}}_{\xi''}$ can be applied. 
Notice also, that for functions with separated variables, $f(x)=f_1(x')f_2(x'')$, $f_1\in L^2(\R^{d'})$, $f_2\in L^2(\R^{d''})$, \eqref{3.3} becomes 
$$
\calG f(k,\xi'')=\calH^{\textsl{sc}}_{\xi''}f_1(k)\calF''f_2(\xi'').
$$
Observe that the structure of $\calG$ expressed in \eqref{3.3} reveals that this transform can be seen as a `twisted' composition of $\calF''$ and 
$\calH^{\textsl{sc}}$.

The definition of the inverse transform $\calG^{-1}$ first requires introducing the space 
\footnote{$\spadesuit$ Here and is some other places we do not distinguish between $\R^{d''}_0$ and $\R^{d''}$.} 
$$ 
L^2(\Gamma):=L^2(\N^{d'}\times \R^{d''},dk\times d\xi''),
$$ 
where  $dk$  is  the counting measure in $\N^{d'}$ and $d\xi''$ is Lebesgue measure in $\R^{d''}$, with norm
$$
\|F\|_{L^2(\Gamma)}=\Big(\sum_{k\in\N^{d'}}\|F(k,\cdot)\|^2_{L^2(\R^{d''})}\Big)^{1/2}=\Big(\sum_{k\in\N^{d'}}\int_{\R^{d''}} |F(k,\xi'')|^2\,d\xi''\Big)^{1/2}.
$$ 
Then $\calG^{-1}$ is defined for $F=F(k,\xi'')\in L^2(\Gamma)$ as a function on $\Rd$ by
\begin{equation}\label{3.4}
\calG^{-1}F(x):=\calF''^{-1}[\calH^{\textsl{sc};-1}_{\xi''}F_{\xi''}(x')](x''), 
\end{equation}
where $F_{\xi''}=F(\cdot,\xi'')$. The correctness of the above definition is a consequence of the fact that for almost every $\xi''\in \R^{d''}_0$, 
$F_{\xi''}\in \ell^2(\N^{d'})$ so that $\calH^{\textsl{sc};-1}_{\xi''}$ can be applied and, moreover, for a.e. $x'\in\R^{d'}$ the function
$\xi''\to \calH^{\textsl{sc};-1}_{\xi''}F_{\xi''}(x')$ is in $L^2(\R^{d''})$ so that $\calF''^{-1}$ can be applied. This is because
$$
\int_{\R^{d''}}\int_{\R^{d'}} |\calH^{\textsl{sc};-1}_{\xi''}F_{\xi''}(x')|^2dx'd\xi''=\int_{\R^{d''}}\Big(\sum_{k\in\N^{d'}}|F_{\xi''}(k)|^2 \Big)d\xi''
=\|F\|^2_{L^2(\Gamma)}<\infty.
$$
Observe that this time $\calG^{-1}$  can be seen as a `twisted' composition of inverses of $\calH^{\textsl{sc}}$ and $\calF''$.

We now prove Plancherel's identity and the inverse formula for the transforms $\calG$ and $\calG^{-1}$.
\begin{theorem} \label{thm:first}
 We have 
\begin{equation}\label{3.5}
\|\calG f\|_{L^2(\Gamma)}=\|f\|_{L^2(\Rd)} \quad and \quad f=\calG^{-1}\big(\calG f\big), \qquad   f\in L^2\big(\Rd),
\end{equation}
and 
\begin{equation}\label{3.6}
\|\calG^{-1} F\|_{L^2(\Rd)}=\|F\|_{L^2(\Gamma)} \quad and \quad F=\calG\big(\calG^{-1}F\big), \qquad F\in L^2(\Gamma).  
\end{equation}
\end{theorem}
\begin{proof}
To verify Plancherel's identity in \eqref{3.5}, for fixed $\xi''\in\R^{d''}_0$ we write
$$
\sum_{k\in\N^{d'}}|\calG(k,\xi'')|^2=\sum_{k\in\N^{d'}}|\calH^{\textsl{sc}}_{\xi''}\big[\calF''f_{x'}(\xi'')\big](k)|^2=\calka|\calF''f_{x'}(\xi'')|^2dx'
$$
and integration over $\xi''$ then gives
$$
\int_{\R^{d''}}\Big(\sum_{k\in\N^{d'}}|\calG(k,\xi'')|^2\Big)d\xi''=\calka\int_{\R^{d''}}|\calF''f_{x'}(\xi'')|^2d\xi''dx'=
\calka\int_{\R^{d''}}|f_{x'}(x'')|^2dx''dx'=\|f\|^2_{L^2(\R^d)}.
$$ 
For the inverse formula in \eqref{3.5} we firstly observe that, by what we just proved, $\calG f\in L^2(\Gamma)$ so $\calG^{-1}$ can be applied.  
Secondly, to avoid a notational collision we write below $(y',\cdot)$ rather than expected $(x',x'')$, and to avoid getting lost at some places we 
shall prompt to which function $\calH^{\textsl{sc}}_{\xi''}$ or $\calF''$ applies.
Now, recalling that by \eqref{3.3}, $(\calG f)_{\xi''}=\calH^{\textsl{sc}}_{\xi''}\big[\calF''f_{x'}(\xi'')\big]$, using \eqref{3.4} 
gives for almost every $y'\in\R^{d'}$ 
\begin{align*} 
\calG^{-1}\big(\calG f\big)(y',\cdot)=\calF''^{-1}\big[\xi''\to\calH^{\textsl{sc};-1}_{\xi''}(\calG_{\a,\b}^\circ f)_{\xi''} (y')\big]
&=\calF''^{-1}\big[\xi''\to\calH^{\textsl{sc};-1}_{\xi''}\big(\calH^{\textsl{sc}}_{\xi''}[x'\to\calF''f_{x'}(\xi'')]\big)(y')\big]\\
&=\calF''^{-1}\big(\calF'' f_{y'} \big)\\ 
&=f_{y'}.
\end{align*} 

To verify Plancherel's identity in \eqref{3.6}  we write
\begin{align*}
\int_{\R^{d''}}\calka|\calG^{-1}F(x',x'')|^2dx'\,dx''&=\int_{\R^{d''}}\calka \big|\calF''^{-1}\big[\calH^{\textsl{sc};-1}_{\xi''}F_{\xi''} (x')\big](x'')\big|^2\,dx''dx'\\
&=\calka\calka\big|\calH^{\textsl{sc};-1}_{\xi''}F_{\xi''} (x')\big|^2d\xi''dx'\\
& = \|F\|^2_{L^2(\Gamma)}.
\end{align*}
For the inversion formula in \eqref{3.6}, observing first that, by what we have just proved, $\calG^{-1}F\in L^2(\Rd)$, we get 
$$
\calG\big(\calG^{-1}F\big)(k,\xi'')
=\calH^{\textsl{sc}}_{\xi''}\Big[\calF''\big(\calG^{-1} F)_{x'}\big)(\xi'')\Big](k)
=\big\langle  \calF''\big((\calG^{-1} F)_{x'}\big)(\xi''),h_{k,|\xi''|}\big\rangle.
$$
The first term in the last brackets is understood, with $\xi''$ fixed, as a function of $x'$ which is
$$
x'\to \calF''\big((\calG^{-1} F)_{x'}\big)(\xi'')   =\calF''\big(\calF''^{-1}[\xi''\to \calH^{\textsl{sc};-1}_{\xi''} F_{\xi''}(r)]\big)(\xi'')=
\calH^{\textsl{sc};-1}_{\xi''}F_{\xi''}(x')=\sum_{l=0}^\infty F(l,\xi'') h_{l,|\xi''|}(x'),
$$
hence the value of the last brackets, where $x'$ is  also the variable of integration, indeed equals $F(k,\xi'')$.
\end{proof}

It remains to conclude that by Theorem~\ref{thm:first}, $\calG^{-1}$ is indeed the inverse of $\calG$, i.e. $\calG^{-1}\circ\calG={\rm Id}_{L^2(\Rd)}$ and
$\calG\circ\calG^{-1}={\rm Id}_{L^2(\Gamma)}$.

\section{Self-adjoint extensions of $G$} \label{sec:self}  
Consider  the real-valued function 
$$
\Theta\colon \N^{d'}\times \R^{d''}_0\to(0,\infty),\qquad \Theta(k,\xi'')=\lambda_k'|\xi''|,\quad \lambda_k'=2|k|+d',
$$
and the corresponding \textit{multiplication operator} $\mathbb M_\Theta$ on $L^2(\Gamma)$ with `maximal' domain, 
\begin{align*}
\D\,\mathbb M_\Theta&=\big\{F\in L^2(\Gamma)\colon \Theta F\in L^2(\Gamma)\big\},\\
 \mathbb M_\Theta F&=\Theta F, \qquad F\in\D\, \mathbb M_\Theta.
\end{align*}
Then a simple argument shows that $\mathbb M_\Theta$ is self-adjoint and nonnegative, and the spectrum of it coincides with $[0,\infty)$, 
which is the set of essential values of $\Theta$. 

We now transfer $\mathbb M_\Theta$ onto $L^2(\Rd)$ through the unitary isomorphism 
$\calG\colon  L^2(\Rd)\to L^2(\Gamma)$ and its inverse $\calG^{-1}\colon L^2(\Gamma) \to L^2(\Rd)$, defining 
\begin{align*}
\D\,\mathbb G:=&\calG^{-1}\big(\D\,\mathbb M_\Theta  \big)= \{f\in L^2\big(\Rd\big)\colon \Theta\,\calG f\in L^2(\Gamma)\},\\
\mathbb G:=&\calG^{-1}\circ\mathbb M_\Theta\circ\calG.
\end{align*}
It follows that $\mathbb G$ is self-adjoint and nonnegative, and the spectrum of $\mathbb G$ is $[0,\infty)$. It remains to verify that $\mathbb G$ 
indeed extends $G$, which is done below. Implicitly, this self-adjoint realization of $G$ is considered in \cite{MS} where, in order to study spectral 
multipliers, a functional calculus corresponding to that described here (see Section~\ref{sec:heat}) was used.  
\begin{proposition} \label{pro:zero}
 $\mathbb G$ is an extension of $G$.
\end{proposition}
\begin{proof} 
With the aid of \eqref{3.2} and Theorem~\ref{thm:first} it follows (notice that $G$ maps $C^\infty_c\big(\Rd\big)$ into itself) that $C^\infty_c\big(\Rd\big)\subset \D\,\mathbb G$. It remains to show that  $\mathbb G\vp=G\vp$ for $\vp\in C^\infty_c\big(\Rd\big)$. To check this it suffices to verify that 
\begin{equation*}
\mathbb{M}(\calG \vp)=\calG(G\vp), 
\end{equation*}
which means verifying that for every $k\in\N^{d'}$ and $\xi''\in\R^{d''}_0$, we have
\begin{equation*}
\Theta(k,\xi'')\calG\vp(k,\xi'')=\int_{\R^d} G\vp(x)\psi_{k,\xi''}(x)\,dx.
\end{equation*}
But this holds because
$$  
\Theta(k,\xi'')\int_{\R^d}   \vp(x)\psi_{k,\xi''}(x)\,dx=\int_{\R^d} \vp(x)G\psi_{k,\xi''}(x)\,dx=\int_{\R^d} G\vp(x)\psi_{k,\xi''}(x)\,dx, 
$$ 
where in the last step a version of Green's formula for $G$ was used; see Lemma~\ref{lem:int}.   
\end{proof} 

This extension has an alternative description.  Let us define the Sobolev-type space adapted to $G$,  
$$
W^2_G(\Rd)=\{f\in L^2(\Rd)\colon Gf\in L^2(\Rd)\}. 
$$
Here and later on $Gf$ is understood in the distributional sense, and $Gf\in L^2(\Rd)$ means that 
\begin{equation}\label{4.1}
\exists h\in L^2(\Rd)\quad\forall \vp\in C^\infty_c(\Rd)\quad \langle G\vp, f\rangle=\langle \vp, h\rangle,\qquad {\rm and\,\,\,then}\quad Gf:=h.
\end{equation}

It is easily checked (see Lemma~\ref{lem:Sobo} for a similar result) that equipped with the scalar product
$$
\langle f,g\rangle_{W^2_G(\Rd)}:=\langle f,g\rangle+\langle Gf,Gg\rangle,
$$
$W^2_G(\Rd)$ becomes a Hilbert space and, since $C^\infty_c(\Rd)\subset W^2_G(\Rd)$, $W^2_G(\Rd)$ is dense in $L^2(\mathbb{R}^d)$. 

\begin{proposition} \label{pro:Sob2}
We have $\D\,\mathbb G=W^2_G(\Rd)$ and $\mathbb Gf=Gf$ for $f\in \D\,\mathbb G$.
\end{proposition}
\begin{proof}
For $\supset$ take $f\in L^2(\Rd)$ such  that $Gf\in L^2(\Rd)$. This means that \eqref{4.1} holds with $h=Gf$. Now, the polarized version of Plancherel's identity \eqref{3.5} from Theorem~\ref{thm:first} implies that also
$$
\langle \calG(G\vp), \calG f\rangle_{L^2(\Gamma)} = \langle \calG\vp, \calG(G f)\rangle_{L^2(\Gamma)}.
$$
But, see \eqref{3.2},
$$
\langle \calG(G\vp), \calG f\rangle_{L^2(\Gamma)}=\langle \Theta \calG\vp, \calG f\rangle_{L^2(\Gamma)}= \langle \calG\vp, \Theta\calG f\rangle_{L^2(\Gamma)},
$$
hence $\calG(G f)=\Theta \calG f\in L^2(\Gamma)$, so $f\in\D\,\mathbb{G}$; we used the fact that $\calG\big(C^\infty_c(\Rd)\big)$ is dense in $L^2(\Gamma)$, 
because $\calG$ is isometrical bijection. For $\subset$, let $f\in \D\,\mathbb{G}$, so that $f\in L^2(\Rd)$ and $\Theta \calG f\in L^2(\Gamma)$. 
Set $h:= \calG^{-1} (\Theta \calG f)$, so that $h\in L^2(\Rd)$.  Repeating the previous arguments but in the opposite order it is easily seen  
that $\langle G\vp, f\rangle= \langle \vp, h\rangle$, which means $G f=h$, hence $f\in W^2_G(\Rd)$, and also $\mathbb Gf=Gf$. 
\end{proof} 

A common way of constructing self-adjoint extensions of differential operators is using sesquilinear forms. This requires introducing 
appropriate Sobolev-type spaces. In the considered case of the Grushin differential operator we begin with an apparently natural way,
see the formulas in \eqref{4.2} and \eqref{4.3}. There is, however, a pitfall caused by singularity of the 
function $x\to|x'|$ along the $(d-d')$-dimensional hyperplane $S'=\{x\in\Rd\colon x'=0\}$. A way of omitting this is simply to remove $S'$ from $\Rd$ 
considering the open set $\Rd_*:=\Rd\setminus S'$, and then to consider $G$ with domain $C^\infty_c(\Rd_*)$.
To distinguish from the previous situation let $G_0$ stand for such operator, i.e. $\D\,G_0=C^\infty_{c}(\Rd_*)\subset L^2(\Rd_*)$. We mention that 
the problem of `removable sets' in constructions of self-adjoint extensions has been studied for other operators in the literature, see, e.g., \cite{BP} and \cite{H}.

Since $S'\subset\Rd$ is of Lebesgue measure zero, $L^2(\Rd_*)$ naturally identifies with $L^2(\mathbb{R}^d)$. Also $C^\infty_c(\Rd_*)$ identifies with 
the space of smooth functions on $\Rd$ with compact supports separated from the closed set $S'$, denoted $C^\infty_{c,*}(\Rd)$; 
 we shall use both identifications with no further mention. Since $C^\infty_{c,*}(\Rd)$ is dense in $L^2(\mathbb{R}^d)$, $G_0$ 
is densely defined. Moreover, $G_0\subset G$, hence $G_0$ is symmetric and nonnegative. 

We now define  the Sobolev space $W^1_{G_0}(\Rd_*)$ (of first order) adapted to $G_0$,  
\begin{equation}\label{4.2}
W^1_{G_0}(\Rd_*)=\{f\in L^2(\mathbb{R}^d)\colon \partial_j f\in L^2(\mathbb{R}^d)\,\, {\rm for}\,\,1\le j\le d'\,\, {\rm and}\,\,  
|x'|\partial_j f\in L^2(\mathbb{R}^d)\,\, {\rm for}\,\, d'+1\le j\le d\}.
\end{equation}
Here and below, for $f\in L^2(\mathbb{R}^d)$, $\partial_j f$ are understood in the sense of distributions on $\Rd_*$  and for $d'+1\le j\le d$,
$|x'|\partial_j f$ means the distribution $\partial_j f$ multiplied by the $C^\infty$ function $|x'|$ on $\Rd_*$; notice that 
$|x'|\partial_j f=\partial_j(|x'|f)$ in the sense of distributions on $\Rd_*$. 

Equipped with the scalar product
$$
\langle f,g\rangle_{G_0}:=\langle f,g\rangle+\sum_{j=1}^{d'}\langle \partial_j f,\partial_j g\rangle+
\sum_{j=d'+1}^{d}\langle |x'|\partial_j f,|x'|\partial_j g\rangle,
$$
$W^1_{G_0}(\Rd_*)$ becomes a Hilbert space (the proof is postponed to Section~\ref{sec:app}, see Lemma~\ref{lem:Sobo}) and, 
since $C^\infty_{c,*}(\Rd)\subset W^1_{G_0}(\Rd_*)$, $W^1_{G_0}(\Rd_*)$ is dense in $L^2(\mathbb{R}^d)$. 

In the next step we define the sesquilinear form
\begin{equation}\label{4.3}
\mathfrak{t}_0[f,g]:=\int_{\Rd}\nabla_{G_0} f\cdot \overline{\nabla_{G_0} g},\qquad f,g\in W^1_{G_0}(\Rd_*),
\end{equation}
where $\nabla_{G_0}:=(\partial_1,\ldots,\partial_{d'},|x'|\partial_{d'+1},\ldots,|x'|\partial_{d})$ is the $G_0$-gradient. Then $\mathfrak{t}_0$ 
is densely defined Hermitian nonnegative and closed form (closedness of $\mathfrak{t}_0$ is just the completeness of the norm generated by 
$\langle\cdot ,\cdot \rangle_{G_0}$). Consequently, it follows from the general theory that the operator $\mathbb G_{0}$ defined by
\begin{align*}
\D\,\mathbb G_0&=\{f\in W^1_{G_0}(\Rd_*)\colon\exists u_f\in L^2(\mathbb{R}^d)\quad \forall g\in W^1_{G_0}(\Rd_*)\quad\mathfrak{t}_0[f,g]=\langle u_f,g\rangle\},\\
\mathbb G_0\,f&=u_f,\qquad f\in \D\, \mathbb G_0,
\end{align*}
is self-adjoint and nonnegative on $L^2(\Rd)$, and extends $G$ which we now prove. 
We add that in the sense of distributions on $\Rd_*$, for $f\in L^2(\Rd)$, $G_0f\in L^2(\Rd)$ means that 
\begin{equation}\label{4.4}
\exists h\in L^2(\Rd)\quad\forall \vp\in C^\infty_{c,*}(\Rd)\quad \langle G\vp, f\rangle=\langle \vp, h\rangle,\qquad {\rm and\,\,\,then}\quad G_0f:=h.
\end{equation}

\begin{proposition} \label{pro:ext2}
$\mathbb G_0$ is an extension of $G_0$. Moreover,
\begin{equation}\label{4.5}
\D\,\mathbb G_0 \subset W^2_{G}(\Rd)
\end{equation}
 and $\mathbb G_0f=G_0f$ for $f\in \D\, \mathbb G_0$.
\end{proposition}
\begin{proof}
We first claim that $C^\infty_{c,*}(\Rd)\subset \D\, \mathbb G_0$ and $ \mathbb G_0\vp=G\vp$ for $\vp\in C^\infty_{c,*}(\Rd)$. For this purpose it suffices 
to check that given $\vp\in C^\infty_{c,*}(\Rd)$, for every $h\in W^1_{G_0}(\Rd_*)$ it holds
\begin{equation}\label{4.6}
\sum_{j=1}^{d'}\langle \partial_j\vp,\partial_j h\rangle+\sum_{j=d'+1}^{d}\langle |x'|\partial_j\vp,|x'|\partial_j h\rangle =\langle G\vp,h\rangle.
\end{equation}
Let $\Omega\subset \Rd_*$ be a bounded open subset containing ${\rm supp}\,\vp$, separated from  $S'$ which is the boundary of $\Rd_*$, with $\partial\Omega$ being sufficiently smooth, say $C^1$. Equivalently, we prove \eqref{4.6} with integration over $\Omega$ that replaces $\Rd$. For this it suffices to check that 
$$
-\int_{\Omega}\partial^2_j\vp\,\overline{h}=\int_{\Omega}\partial_j\vp\overline{\partial_j h}\quad {\rm and}\quad 
-\int_{\Omega}|x'|^2\partial^2_j\vp\,\overline{h}=\int_{\Omega}|x'|^2\partial_j\vp\overline{\partial_j h}.
$$
But $h\in  W^1_{G_0}(\Rd_*)$ which implies that $h\in H^1(\Omega)$, where $H^1(\Omega)$ denotes the Euclidean Sobolev space on $\Omega$. 
(We have $\partial_j (h|_\Omega)=(\partial_j h)|_\Omega)$ and since $|x'|\partial_j h\in L^2(\mathbb{R}^d)$, also $\partial_j (h|_\Omega)\in L^2(\Omega)$ 
for $d'+1\le j\le d$.) Thus the two above identities follow from Gauss'  formula \eqref{Gauss}.

To check \eqref{4.5} note that $f\in \D\,\mathbb{G}_0$ means that $f\in W_{G_0}^1(\Rd)$ and there is $u_f\in L^2(\Rd)$ such that for every $h\in W_{G_0}^1(\Rd)$ 
we have $\mathfrak{t}_0[h,f]=\langle h,u_f\rangle$. In particular, taking  $h=\vp\in C^\infty_{c,*}(\Rd)$ gives 
$$
\forall \vp\in C^\infty_{c,*}(\Rd)\qquad\sum_{j=1}^{d'}\langle \partial_j\vp,\partial_j f\rangle+\sum_{j=d'+1}^{d}\langle |x'|\partial_j\vp,|x'|\partial_j f\rangle
=\langle \vp,u_f\rangle.
$$
But the left-hand side of the above identity equals $\langle G\vp,f\rangle$, see \eqref{4.6}, which implies that $G_0f=u_f$, hence  $f\in W_G^2(\Rd)$ and
$\mathbb G_0f=G_0f$.
\end{proof}

Obviously, $\mathbb G_0=\mathbb G$ cannot be a consequence of essential self-adjointness of $G$, but the identity is indeed true.

\begin{proposition} \label{pro:oper2}
We have $\mathbb G_0=\mathbb G$.
\end{proposition}
\begin{proof}
By Propositions~\ref{pro:ext2} and \ref{pro:Sob2} we know that $\D\,\mathbb G_0 \subset \D\,\mathbb{G}$  and, in addition, by comparing \eqref{4.1} 
and \eqref{4.4} we see that for every $\vp\in C^\infty_{c,*}(\Rd)$ we have $\langle \vp,G_0f\rangle=\langle\vp,Gf\rangle$. Density of $C^\infty_{c,*}(\Rd)$ 
in  $L^2(\Rd)$ then shows that $G_0f=Gf$. Thus we have the inclusion $\mathbb{G}_0\subset \mathbb{G}$. This is enough due to the well-known maximality 
property of self-adjoint operators.
\end{proof}

Another way of using sesquilinear forms is to look at $G$ through the identity
$$
G=\sum_{j=1}^{d'}\partial_j^2+ \sum_{j=1}^{d'} \sum_{m=1}^{d''}x_j^2\partial_{d'+m}^2
$$
and consider, instead of $\nabla_{G_0}$, the gradient (cf. the vector fields $X_j$ and $X_{j,m}$ at the end of Section~\ref{sec:heat})
$$
\nabla_{G}=\big((\partial_j)_{1\le j\le d'}, (x_j'\partial_{d'+m})_{1\le j\le d',\, 1\le m\le d''}\big)
$$ 
(we owe this suggestion to Allessio Martini to whom we are grateful). The accompanying Sobolev-type space is now defined as 
\begin{equation*}
W^1_{G}(\Rd)=\{f\in L^2(\mathbb{R}^d)\colon \partial_j f,\,x_j'\partial_{d'+m} f\in L^2(\mathbb{R}^d)\,\, \,\,  {\rm for}\,\,1\le j\le d',\,\, 1\le m\le d''\},
\end{equation*}
and it is standard to verify that the norm generated by the inner product 
$$
\langle f,g\rangle_{W^1_{G}(\Rd)}:=\langle f,g\rangle+\sum_{j=1}^{d'}\langle \partial_j f,\partial_j g\rangle+
\sum_{j=1}^{d'}\sum_{m=1}^{d''}\langle x_j'\partial_{d'+m} f,x_j'\partial_{d'+m} g\rangle,
$$
is complete in $W^1_{G}(\Rd)$. Therefore, the sesquilinear form 
\begin{equation*}
\mathfrak{t}[f,g]:=\int_{\Rd}\nabla_{G} f\cdot \overline{\nabla_{G} g},\qquad f,g\in W^1_{G}(\Rd),
\end{equation*}
(here `$\cdot$' stands for the inner product in $\R^{d'(1+d'')}$) is Hermitian nonnegative and closed. Consequently, the associated operator, call it 
for a moment $\hat{\bbG}$, is self-adjoint and nonnegative on $L^2(\Rd)$. Verification that $\hat{\bbG}$ extends $G$, $\D(G)=C^\infty_c(\Rd)$, goes 
along the lines of the relevant reasoning for $G_0$. Finally, by essential self-adjointness of $G$, $\hat{\bbG}=\bbG$.

We conclude this section by a discussion of some group invariance and homogeneity properties of $G$ inherited then by $\mathbb G$. Let us begin with rotational invariance. The product of orthogonal groups $O(d')\times O(d'')$ naturally acts on functions on $\Rd$. Namely, given $g=(g',g'')\in O(d')\times O(d'')$ let
$$
T_{g'}f(x)=f(g'^{-1}x',x'')\quad {\rm and}\quad T_{g''}f(x)=f(x',g''^{-1}x''),
$$
and then define $T_g=T_{g'}\circ T_{g''}=T_{g''}\circ T_{g'}$, i.e., 
$$
T_g f(x)=f(g'^{-1} x',g''^{-1} x''), \qquad x\in\R^d.
$$
The structure of $G$ suggests that for $g\in O(d')\times O(d'')$, $T_g$ should commute with $G$. For this it suffices to verify that $T_{g'}$ 
and $T_{g''}$ commute with $G$. But $T_{g'}$ commutes with $\Delta_{x'}$ and (trivially) with $\Delta_{x''}$, so
\begin{align*}
\big(G\circ T_{g'^{-1}}\big)f(x',x'')=G\big(f(T_{g'}x',x'')\big)
&=\Delta_{x'}\big(f(T_{g'}x',x'')\big)+|T_{g'}x'|^2\Delta_{x''}\big(f(T_{g'}x',x'')\big)\\
&=\big(\Delta_{x'}f\big)(T_{g'}x',x'')+|x'|^2\big(\Delta_{x''} f\big)(T_{g'}x',x'')\\
&=\big(T_{g'^{-1}}\circ G\big)f(x',x'').
\end{align*}
Analogous calculation with the aid of $\Delta_{x''}\circ T_{g''}=T_{g''}\circ \Delta_{x''}$ (and trivial commutation of $T_{g'}$ with $\Delta_{x'}$) 
shows $G\circ T_{g''}=T_{g''}\circ G$. 

The second property concerns homogeneity of $G$ with respect to the family of \textit{parabolic dilations} of $\Rd$, $\delta_r(x)=(rx',r^2x'')$, then 
involved in parabolic dilations  of functions on $\Rd$, 
$$
\rho_r f(x)=r^{(d'+2d'')/2}f(rx',r^2x''), \qquad r>0. 
$$
A short calculation shows that $G$ is homogeneous of degree 2 in the sense that 
$$
G\circ \rho_r=r^2 (\rho_r\circ G).
$$
It is convenient to denote $D:=d'+2d''$ and consider $D$ as the \textit{homogeneous dimension} for $G$.

That the two discussed properties are inherited by $\mathbb G$ should  not come as a surprise but certainly it requires a formal proof. From now on we 
consider $T_g$ and $\rho_r$ as operators on $L^2(\Rd)$. Obviously, $T_g$ and $\rho_r$ are unitary automorphisms of $L^2(\Rd)$, and $(T_g)^*=T_g^{-1}=T_{g^{-1}}$, 
$(\rho_r)^*=\rho_r^{-1}=\rho_{1/r}$. It may be also easily checked that $T_g$ and $\rho_r$ map $W^2_{G}(\Rd)$ into itself, the restrictions $\check T_g$ and 
$\check\rho_r$ of $T_g$ and $\rho_r$ to $W^2_{G}(\Rd)$ are unitary automorphisms of $W^2_{G}(\Rd)$, and $(\check T_g)^*=\check T_{g^{-1}}$,  
$(\check\rho_r)^*=\check\rho_{1/r}$. Here the superscript `$*$' refers to the adjoint operator in the Hilbert space $W^2_{G}(\Rd)$.

\begin{proposition}\label{pro:one} 
For every $g\in O(d')\times O(d'')$  and $r>0$ we have 
$$
\mathbb G\circ T_g=T_g\circ \mathbb G \qquad and \qquad \mathbb G\circ \rho_r=(r^2\rho_r)\circ \mathbb G.
$$
\end{proposition}
\begin{proof} The proof relies on Proposition~\ref{pro:Sob2}, which says that $\D\,\mathbb G=W^2_G(\Rd)$ and $\mathbb Gf=Gf$ for $f\in \D\,\mathbb G$. 
For the first identity note that the properties of the restriction of $T_g$ to $W^2_{G}(\Rd)$ imply 
$\D(\mathbb G\circ T_g)=\D\,\mathbb G=\D(T_g\circ \mathbb G)$. Now, fix $f\in \D\,\mathbb G=W^2_G(\Rd)$ so that $\mathbb Gf=Gf$. This means that 
\begin{equation*}
\forall \vp\in C^\infty_c(\Rd)\qquad \langle G(T_{g^{-1}}\vp), f \rangle=\langle T_{g^{-1}}\vp, Gf \rangle,
\end{equation*}
and by a change of variable,
$$
\forall \vp\in C^\infty_c(\Rd)\qquad \langle G\vp, T_{g}f \rangle=\langle \vp,T_{g}( Gf)\rangle.
$$
But this means that $T_g f\in W^2_{G}(\Rd)$ and $\mathbb G(T_g f)=T_g(\mathbb G f)$. The second identity from the statement is proved using analogous arguments.
\end{proof}

\section{Heat kernel for $\mathbb G$} \label{sec:heat}  
The functional calculus, an important ingredient of the spectral theory of self-adjoint operators, for a multiplication operator 
on an $L^2$ space has a simple description. In the specific case of the multiplication operator $\mathbb{M}_{\Theta}$, given a Borel function 
$\Phi\colon[0,\infty)\to\C$ (recall that $\sigma(\mathbb{M}_\Theta)=[0,\infty)$) we set 
$$ 
\Phi(\mathbb{M}_{\Theta}):=M_{\Phi\circ\Theta}.
$$
Then the functional calculus is understood  as the mapping $\Phi\to M_{\Phi\circ\Theta}$ assigning to $\Phi$  the multiplication operator by 
$\Phi\circ\Theta$ on $L^2(\Gamma)$. Consequently, the functional calculus for the self-adjoint operator $\mathbb G$, which
is unitarily equivalent to $\mathbb{M}_{\Theta}$, is the realization of the functional calculus for $\mathbb M_\Theta$ mapped by $\calG$ onto 
$L^2\big(\Rd\big)$. Precisely, this means the mapping 
\begin{equation}\label{5.1}
\Phi\to \Phi(\mathbb G):=\calG^{-1}\circ M_{\Phi\circ\Theta}\circ\calG,
\end{equation}
so that the domain of $\Phi(\mathbb G)$ is 
$$
\D\,\Phi(\mathbb G)=\calG^{-1}(\D\,M_{\Phi\circ\Theta} ) =\{f\in L^2\big(\Rd\big)\colon \big(\Phi\circ\Theta\big)\calG f\in L^2(\Gamma)\}.
$$
Obviously, $\Phi(\mathbb G)$ inherits all the spectral properties of $\Phi(\mathbb M_{\Theta})$. 

Now we apply the above comments to the bounded functions $\Phi_t(u)=e^{-tu}$, $u\ge0$, where $t>0$ is a parameter. The corresponding family 
$\{\exp(-t \mathbb G)\}_{t>0}$ is the one-parameter semigroup of operators bounded on $L^2(\Rd)$, called the heat semigroup for $\mathbb G$. 
We prove that these operators are integral operators; the family of the corresponding kernels $\{p_t\}_{t>0}$ is then called the heat kernel for $\mathbb G$. 

We mention that up to some normalization constants caused by a slightly different normalization of the Grushin operator, our formula \eqref{5.2} 
coincides with that obtained by Garofalo and Tralli in \cite[Theorem~3.4]{GT}.
\footnote{\textcolor[rgb]{1,0,0}{$\heartsuit$} It is helpful to point out that for $x',y'\in \R^{d'}$ and $u>0$ it holds
$$
\coth u\,\big(|x'|^2+|y'|^2-2x'\cdot y'\,{\rm sech\,u}\big)=\frac12\big(|x'+y'|^2\tanh(u/2) +|x'-y'|^2\coth(u/2)\big).
$$} 
\begin{theorem} \label{thm:heat}
For every $t>0$, $\exp(-t\mathbb G)$ is an integral operator with kernel 
\begin{align}
&\qquad \qquad \qquad p_t(x,y)=(2\pi t)^{-D/2}\times\label{5.2}\\
&\int_{\R^{d''}} e^{\textrm{i}\frac 1t (x''-y'') \cdot\xi''}\exp\Big(-\frac{|\xi''|}{4t}\big(|x'+y'|^2\tanh|\xi''| +|x'-y'|^2\coth|\xi''|\big)\Big)
\Big(\frac{|\xi''|}{\sinh 2|\xi''|}\Big)^{d'/2}\,d\xi'',\nonumber
\end{align}
that is, for $f\in L^2(\Rd)$,
\begin{equation}\label{el2}
\exp(-t \mathbb G)f(x)=\int_{\Rd} p_t(x,y)f(y)\,dy.
\end{equation}
\end{theorem}
\begin{proof}
Using \eqref{5.1} we obtain $\exp(-t\mathbb G)=\calG^{-1}\circ \mathbb{M}_{\Phi_t\circ\Theta}\circ\calG$, i.e. for $f\in L^2(\Rd)$,
\begin{equation}\label{5.3}  
\exp(-t\mathbb G)f(x)=\calF''^{-1}\Big[\sum_{k\in\N^{d'}}e^{-t\lambda_k'|\xi''|}\calG f(k,\xi'')h_{k,|\xi''|} (x') \Big] (x'')
\equiv\calF''^{-1}\Big[S_{t,x'}f(\xi'') \Big] (x'').
\end{equation}
Note that by \eqref{3.5}, $\calG f\in L^2(\Gamma)$, hence for almost every $x'$, the function of $\xi''$ inside the square brackets in \eqref{5.3}
is in $L^2(\R^{d''})$ so that $\calF''^{-1}$ can be applied.

Assuming temporarily that $f=\vp\in C^\infty_c(\Rd)$ with ${\rm supp}\, \vp\subset\{y\in\Rd\colon |y'|\le r,\, |y''|\le r\}$ for some $r>0$, we evaluate 
the sum in \eqref{5.3}. Using   the integral representation \eqref{3.1}  of $\calG \vp(k,\xi'')$ gives 
\begin{equation*}%
S_{t,x'}\vp(\xi'')=
(2\pi)^{-d''/2}\int_{\R^{d''}}e^{-\textrm{i}y''\cdot\xi''}\int_{\R^{d'}}\vp(y',y'')T_{t,x',\xi''}(y')\,dy'dy''
\end{equation*}
with
\begin{equation*}
T_{t,x',\xi''}(y')=\sum_{k\in\N^{d'}}e^{-t\lambda_k'|\xi''|}h_{k,|\xi''|}(x') h_{k,|\xi''|}(y').
\end{equation*}
Exchanging summation with integration was possible since for $t,x',\xi''$ fixed, we have
\begin{align*}
&\sum_{n=0}^\infty e^{-t|\xi''|(2n+d')}\sum_{|k|=n} \int_{|y'|\le r} \int_{|y''|\le r} \big|h_{k}(|\xi''|^{1/2}x')h_{k}(|\xi''|^{1/2} y')\big|\,dy'\,dy''\\
&\lesssim \sum_{n=0}^\infty e^{-2t|\xi''| n}\sum_{|k|=n} \big|h_{k}(|\xi''|^{1/2}x')\big|\| h_{k} \|_{L^1(\R^{d'})}
<\infty,
\end{align*}  
and this is enough for our purposes; we used the fact that by \eqref{fiu} the latter sum  grows polynomially in $n\to\infty$.

Now, Mehler's formula for the Hermite polynomials, \cite[p.\,380]{Sz}, adapted to the one-dimensional Hermite functions $\{h_k\}_{k\in\N}$ reads for $0<r<1$, 
$$
\sum_{k=0}^\infty r^kh_k(u)h_k(v)=\frac1{\pi^{1/2}(1-r^2)^{1/2}}\exp\Big(-\frac12\frac{1+r^2}{1-r^2}(u^2+v^2)+\frac{2r}{1-r^2}uv\Big), \quad u,v\in\R.
$$
Then a simple computation shows that the heat kernel for the Hermite operator on $L^2(\R^{d'})$,
\begin{equation*}
G_t(x',y')=\sum_{n=0}^\infty e^{-t(2n+d')}\sum_{|k|=n}h_k(x')h_k(y'),
\end{equation*}
has the closed form (see \cite[(1.4)]{ST2} which is a more symmetric version of the formula in \cite[p.\,453]{ST})
$$
G_t(x',y')=(2\pi\sinh 2t)^{-d'/2}\exp\Big(-\frac14\big(|x'+y'|^2 \tanh t+|x'-y'|^2 \coth t\big)\Big).
$$
Thus, in the scaled setting 
\begin{align}\label{www} %
 T_{t,x',\xi''}(y') 
&\,\,=|\xi''|^{d'/2} \sum_{n=0}^\infty e^{-t|\xi''|(2n+d')}\sum_{|k|=n}h_k(|\xi''|^{1/2}x')h_k(|\xi''|^{1/2}y')\\
&\,\,=|\xi''|^{d'/2} G_{t|\xi''|}(|\xi''|^{1/2}x',|\xi''|^{1/2}y')\nonumber\\
&\,\,=\Big(\frac{|\xi''|}{2\pi\sinh 2t|\xi''|}\Big)^{d'/2}\exp\Big(-\frac{|\xi''|}4\big(|x'+y'|^2 \tanh t|\xi''|+ |x'-y'|^2\coth t|\xi''|\big)\Big).\nonumber
\end{align}
Consequently, still for $\vp\in C^\infty_c(\Rd)$, 
\begin{align*} 
\exp(-t\mathbb G)\vp(x)=\calF''^{-1} S_{t,x'}\vp(x'')
&=(2\pi)^{-d''/2} \int_{\R^{d''}} e^{\textrm{i}x''\cdot\xi''}S_{t,x'}\vp(\xi'')\,d\xi''\\
&=(2\pi)^{-d''}\int_{\R^{d''}}e^{\textrm{i}x''\cdot\xi''}\int_{\R^{d''}}e^{-\textrm{i}y''\cdot\xi''}\int_{\R^{d'}}\vp(y',y'')T_{t,x',\xi''}(y')\,dy'dy''\,d\xi''\\
&=\int_{\R^{d}}p_t(x,y)\vp(y)\,dy,
\end{align*} 
where
\begin{align} 
& p_t(x,y)=(2\pi)^{-(d'+2d'')/2}\int_{\R^{d''}} e^{\textrm{i}(x''-y'')\cdot\xi''}T_{t,x',\xi''}(y')  \,d\xi''=(2\pi)^{-D/2}\times   \label{uuu}\\ 
&\int_{\R^{d''}} e^{\textrm{i}(x''-y'')\cdot\xi''}\exp\Big(-\frac{|\xi''|}4 \big(|x'+y'|^2\tanh t|\xi''| +|x'-y'|^2\coth t|\xi''|\big)\Big)
\Big(\frac{|\xi''|}{\sinh 2t|\xi''|}\Big)^{d'/2}\,d\xi''.\nonumber
\end{align} 
The change of variable, $t\xi''\to\xi''$, then gives the form of $p_t(x,y)$ as in \eqref{5.2}. Exchanging the order of integration was possible since 
for $t>0$ and $x'$ fixed,
$$
\int_{\R^{d''}}\int_{\R^{d''}}\int_{\R^{d'}} |\vp(y',y'')T_{t,x',\xi''}(y')|\,dy'dy''\,d\xi''
\lesssim r^d \int_{\R^{d''}}\Big(\frac{|\xi''|}{\sinh 2t|\xi''|}\Big)^{d'/2} \,d\xi''<\infty.
$$ 

For general $f\in L^2(\Rd)$ we argue as follows. Let $\{\vp_n\}\subset C^\infty_c(\Rd)$ be such that $\vp_n\to f$ in $L^2(\Rd)$, hence also 
$\exp(-t \mathbb G)\vp_n\to \exp(-t \mathbb G)f$ in $L^2(\Rd)$. With $t>0$ fixed, define
$$
Tg(x):=\calka\calka p_t (x,y)g(y)\,dy, \qquad g\in L^2(\Rd), \quad x\in\Rd.
$$
By Proposition~\ref{pro:add} this definition is correct. By what we have already proved, combined with Proposition~\ref{pro:add} and Schwarz` inequality, gives
$$
\exp(-t \mathbb G)\vp_n(x)=T \vp_n(x)\to Tf(x), \qquad n\to\infty,
$$
for every $x\in\Rd$. It follows that $\exp(-t \mathbb G)f(x)=Tf(x)$, $x$-a.e.
\end{proof}

To provide an alternative expression for $p_t(x,y)$ let $J_\a$ stand for the Bessel function of the first kind of order $\a$, see \cite[Section 5]{Leb}. 
Recall that for radial functions the Fourier transform may be replaced by the Hankel transform. More precisely, and in the specific case of $\R^{d''}$, 
if $F(\xi'')=f(|\xi''|)$, then $\calF'' F(x'')=\calH_{(d''-2)/2}f(|x''|)$, where for general $\a>-1$,  $\calH_\a$ denotes the Hankel transform given by
$$
\calH_\a f(r)=\int_0^\infty \frac{J_\a(ru)}{(ru)^\a}f(u)\,u^{2\a+1}du,  \qquad r>0.
$$

Now, letting
$$
\psi_{t,x',y'}(\tau)=\exp\Big(-\frac{\tau }{4t} \big(|x'+y'|^2\tanh \tau+|x'-y'|^2\coth\tau\big)\Big)\Big(\frac{\tau}{\sinh 2\tau}\Big)^{d'/2},\qquad \tau>0,
$$
observe that by \eqref{5.2}
\begin{equation}\label{5.7}
p_t(x,y)= (2\pi)^{d''/2}(2\pi t)^{-D/2} \calF''(\Psi_{t,x',y'})\big(t^{-1}(y''-x'')\big),
\end{equation}
where  $\Psi_{t,x',y'}(\xi'')=\psi_{t,x',y'}(|\xi''|)$ is a radial function on $\R^{d''}$. Hence,  
$$
\calF''(\Psi_{t,x',y'})\big(t^{-1}(y''-x'')\big)= \calH_{(d''-2)/2}\psi_{t,x',y'}\big(|y''-x''|/t\big),
$$
and the following holds.
\begin{corollary}\label{cor:1}
We have
$$
p_t(x,y)= (2\pi)^{d''/2}(2\pi t)^{-D/2}\int_0^\infty \frac{J_{(d''-2)/2}(\tau|y''-x''|/t)}{(\tau|y''-x''|/t)^{(d''-2)/2}}\psi_{t,x',y'}(\tau)\,\tau^{d''-1}d\tau.
$$
\end{corollary}

We now comment on elementary properties of the heat kernel. First, due to self-adjointness of $\exp(-t\mathbb G)$, for every $t>0$, $p_t(x,y)$ is 
symmetric in $x,y\in\Rd$, as it is directly seen in \eqref{5.2}. Next, recalling that $T_g$ and $\rho_r$ are considered as bounded operators on $L^2(\Rd)$,  
 and using the identities  from Proposition~\ref{pro:one}, it follows that for every $g\in O(d')\times O(d'')$, 
$r>0$, and $t>0$ it holds 
\begin{equation}\label{5.8}
T_g\circ\exp(-t\mathbb G)=\exp(-t\mathbb G)\circ T_g \quad {\rm and}\quad  \exp(-t\mathbb G)\circ \rho_r=\rho_r\circ \exp(-tr^2\mathbb G).
\end{equation}
These two equalities  are then reflected in the identities for the heat kernel, namely
$$
p_t(x,y)=p_t(gx,gy) \quad {\rm and}\quad p_t(x,y) = r^{-D}p_{t/r^2}\big(\delta_{1/r}x,\delta_{1/r}y\big),
$$
and are, a posteriori, easily seen in \eqref{5.2}.

Coming back to \eqref{5.8}, the first identity is a consequence of the well known commutation property for the functional calculus: if $A$ is 
a self-adjoint operator on a Hilbert space $H$ and $B$ is a bounded operator on $H$ such that $BA\subset AB$, then also $B\Psi(A)\subset \Psi(A)B$, 
for any Borel function $\Psi$ on $\sigma(A)$. In addition, if $\Psi$ is bounded, then $\Psi(A)$ is a bounded operator and the last inclusion becomes the identity. 
The second follows from the following  two-operator version of the above. Namely, in the same setting, if $A_1$ and $A_2$ are self-adjoint operators  
such that $BA_1\subset A_2B$, then also $B\Psi(A_1)\subset \Psi(A_2)B$, for any Borel function $\Psi$ on $\mathbb R$. Again, if $\Psi$ 
is bounded, then the last inclusion becomes the identity. We apply this version to $A_1=r^2\bbG$, $A_2=\bbG$, $B=\rho_r$, and $\Psi(u)=\Psi_t(u)=e^{-tu}$, 
$u>0$, taking into account that $\rho_r\circ(r^2\bbG)=\bbG\circ\rho_r$, which is a modified version of the second identity in Proposition~\ref{pro:one}. 
For more comments on this version of commutation property we refer the reader to \cite[p.\,180]{MSt}.

Heuristically, for any fixed $y\in\Rd$, $p_t(x,y)$ should satisfy the heat equation for the Grushin operator. Below we deliver the exact proof 
(which is more painful than one could expect) of this.
\begin{proposition} \label{pro:heat}
Given $y\in\Rd$ let $v(x,t)=p_t(x,y)$. Then $v$ satisfies the heat equation
\begin{equation}\label{5.9}
(\partial_t+G_x)v(x,t)=0, \qquad x\in\Rd,\quad t>0. 
\end{equation}
Moreover, $p_t(x,y)$ is $C^\infty$ on $\Rd\times\Rd\times(0,\infty)$.
\end{proposition}
\begin{proof} 
For \eqref{5.9} it is convenient to use \eqref{uuu}, the integral representation of $p_t(x,y)$ (with $x''-y''$ replaced by $y''-x''$, this 
does not make any difference) together with  \eqref{www}, the series representation of $T_{t,x',\xi''}(y')$. The first step toward \eqref{5.9}, 
justification that entering with $\partial_t$ and $G_x$ under the integral sign in \eqref{uuu} is allowed, is technically involved hence it is 
postponed to Lemma~\ref{lem:last}. Assuming this step was made we write
$$ 
\partial_t\Big(e^{-\textrm{i}(x''-y'')\cdot\xi''} T_{t,x',\xi''}(y')\Big)=e^{-\textrm{i}(x''-y'')\cdot\xi''}\sum_{k\in\N^{d'}}(-\lambda_k'|\xi''|)e^{-t\lambda_k'|\xi''|}h_{k,|\xi''|}(x')  h_{k,|\xi''|}(y')
$$
and, with the aid of Lemma~\ref{lem:second},
\begin{align*}
G_x\Big(e^{-\textrm{i}(x''-y'')\cdot\xi''}T_{t,x',\xi''}(y')\Big)
&=e^{\textrm{i}y''\cdot\xi''}\sum_{k\in\N^{d'}} G_x\big(e^{-\textrm{i}x''\cdot\xi''}h_{k,|\xi''|}(x') \big) h_{k,|\xi''|}(y')  \\
&=e^{-\textrm{i}(x''-y'')\cdot\xi''}\sum_{k\in\N^{d'}}\lambda_k'|\xi''|e^{-t\lambda_k'|\xi''|}h_{k,|\xi''|}(x')  h_{k,|\xi''|}(y'),
\end{align*}
so that \eqref{5.9} follows. Term by term differentiation of the relevant series, i.e. exchanging $\partial_t,\partial_{x_j}$, $\partial_{x_j}^2$ with 
the sum representing $T_{t,x',\xi''}(y')$ in \eqref{www} is allowed. This can be read off from the proof of \cite[Proposition~2.5]{ST} (in the setting 
of $\Rd$ in place of $\R^{d'}$). Here the fact, implied by \eqref{fiu}, that the quantity $\sum_{|k|=n}\big|h_k(|\xi''|^{1/2}x')h_k(|\xi''|^{1/2}y')\big|$ 
growth polynomially in $n\to \infty$ for $y'\in\R^{d'}$ fixed and $|x'|\le A$, $A>0$, is crucial.

For the proof of the second part of the proposition we use the fact that the operator $\partial_t+G_x$, the evolutive counterpart of $G$, is hypoelliptic. Indeed, denoting $X_j=\partial_j$ and $X_{j,m}=x_j\partial_{d'+m}$ for $j=1,\ldots,d'$, $m=1,\ldots,d''$, and $X_0=-\partial_t$, we have
$$
-(\partial_t+G_x)=X_0+\sum_1^{d'} X_j^2\,+\,\sum_1^{d'}\sum_1^{d''} X_{j,m}^2,
$$
and H\"ormander's criterion on $\Rd \times (0,\infty)$ is satisfied because 
$\partial_{d'+m}=[X_{d'},X_{d', d'+m}]$ for $m=1,\ldots,d''$. 

Thus, by the general theory it follows that $p_t(x,y)$ is a $C^\infty$ function on $\Rd\times(0,\infty)$ for $y\in\Rd$ fixed, 
and the symmetry of $p_t(x,y)$ in $x,y\in\Rd$ then  shows  $C^\infty$ on $\Rd\times\Rd\times(0,\infty)$.
\end{proof} 


\section{Auxiliary results} \label{sec:app}

As declared earlier, in this section we collect proofs of auxiliary results used earlier. We begin with result applied in the proof of Theorem~\ref{thm:heat} 
to check that \eqref{el2} holds for $f\in L^2(\Rd)$ with $p_t$ given by \eqref{5.2}. 
\begin{proposition} \label{pro:add}
For every $t>0$ and $x\in\Rd$ we have $p_t\big(x,\cdot)\in L^2(\Rd)$.
\end{proposition}
\begin{proof}
Using \eqref{5.7} and the notation preceding Corollary~\ref{cor:1} we get 
\begin{align*}
&\qquad (2\pi)^{d}t^{d''}\int_{\Rd}|p_t(x,y)|^2dy\\
&=t^{-d}\int_{\R^{d'}}\int_{\R^{d''}}|\calF''(\Psi_{t,x',y'})\big(t^{-1}(y''-x'')\big)|^2dy''dy'\\
&=\int_{\R^{d'}}\int_{\R^{d''}}|\Psi_{t,x',y'})(\xi'')|^2d\xi''dy'\\
&=\int_{\R^{d''}}\Big(\frac{|\xi''|}{\sinh 2|\xi''|}\Big)^{d'}\int_{\R^{d'}}\exp\Big(-\frac{|\xi''|}2 \big(|x'+y'|^2\tanh |\xi''|+
|x'-y'|^2\coth|\xi''|\big)\Big) dy'd\xi''\\
&\le\int_{\R^{d''}}\Big(\frac{|\xi''|}{\sinh 2|\xi''|}\Big)^{d'}\int_{\R^{d'}}\exp\Big(-\frac{|\xi''|}2|y'|^2\coth|\xi''|\Big) dy'd\xi''\\
&\lesssim \int_{\R^{d'}}\exp\big(-|y'|^2\big) dy'\int_{\R^{d''}}\Big(\frac{|\xi''|}{\sinh|\xi''|}\Big)^{d'/2} (\cosh|\xi''|)^{-3/2} d\xi''<\infty.
\end{align*}
In the last step a change of variable was made and the formula $\sinh 2a=2\sinh a\cosh a$ was used. 
\end{proof}

The second auxiliary result is Green's formula adapted to $G$. Recall that in the (classical) setting of the Laplacian,  for an open bounded 
$\Omega\subset\Rd$, $d\ge1$, with boundary $\partial\Omega$ of class $C^1$ (when $d\ge2$) and for  $f,g\in H^1(\Omega)$, the classical Sobolev 
space on $\Omega$, it holds for $j=1,\ldots,d$ (Gauss' formula)
\begin{equation}\label{Gauss}
-\int_\Omega (\partial_j f)\overline{g}\,dx=\int_\Omega  f\overline{\partial_j g}\,dx-\int_{\partial \Omega}f\overline{g}\nu_j\,d\sigma, 
\end{equation}
and for $\partial\Omega$ of class $C^2$ and  $f,g\in H^2(\Omega)$ it holds (Green's formula)
\begin{equation}\label{Green2} 
\int_\Omega (-\Delta_d f)\overline{g}\,dx=\int_\Omega  f\overline{(-\Delta_d g)}\,dx+\int_{\partial \Omega}\Big(f\overline{\partial_{\nu}g}-\partial_{\nu}f\,\overline{g}  \Big)\,d\sigma. 
\end{equation}
Here $d\sigma$ is the surface measure on $\partial \Omega$, and $\nu=\nu(x)$ is the outward unit normal vector at $x\in \partial \Omega$. See, for instance,  
\cite[Theorem D.9\,p.\,408]{Sch}. 

In fact we needed a version of \eqref{Green2} with much weaker assumptions on the involved functions. Notice that Lemma~\ref{lem:int} shows, in particular, the  symmetry of $G$ on $C^\infty_c(\Rd)$.
\footnote{$\clubsuit$  An argument analogous to that from the proof of Lemma~\ref{lem:int} shows nonnegativity of $G$ on $C^\infty_c(\Rd)$.} 
\begin{lemma} \label{lem:int}
Let $\vp\in C^\infty_c(\Rd)$ and $\psi\in C^\infty(\Rd)$. Then 
\begin{equation}\label{6.3}
\int_{\Rd} G\vp(x)\overline{\psi(x)}\,dx=\int_{\Rd} \vp(x)\overline{G\psi(x)}\,dx.
\end{equation} 
\end{lemma}
\begin{proof}
Assume that ${\rm supp}\, \vp\subset A_r:=\{y\in\Rd\colon |x'|\le r,\, |x''|\le r\}=\overline{B_r'}\times \overline{B_r''}$ so that the integration 
over $\Rd$ can be replaced by integration over $A_r$. Obviously, $\vp_{x''}\in C_c^\infty(\R^{d'})$ with ${\rm supp}\,\vp_{x''}\subset \overline{B_r'}$ 
for $x''\in\R^{d''}$, and analogously for $\vp_{x'}$, $x'\in\R^{d'}$, and similarly for $\psi_{x'}$ and $\psi_{x''}$. Then we have 
\begin{align*}
&\int_{A_r}\big(-\Delta_{x'}-|x'|^2\Delta_{x''}\big)\vp(x)\overline{\psi(x)}\,dx\\
&=\int_{B_{r}''}\int_{B_{r}'}\big(-\Delta_{x'}\vp_{x''}\big)(x')\overline{\psi_{x''}(x')}\,dx'dx''- \int_{B_{r}'}|x'|^2\int_{B_{r}''}
\big(-\Delta_{x''}\vp_{x'}\big)(x'')\overline{\psi_{x'}(x'')}\, dx''dx'\\
&=\int_{B_{r}''}\int_{B_{r}'}\vp_{x''}(x')\big(-\Delta_{x'}\overline{\psi_{x''}\big)(x')}\,dx'dx''- \int_{B_{r'}}|x'|^2\int_{B_{r}''}
\vp_{x'}(x'')\overline{\big(-\Delta_{x''}\psi_{x'}\big)(x'')}\,dx''dx'\\
&=\int_{A_r}\vp(x)\overline{\big(-\Delta_{x'}-|x'|^2\Delta_{x''}\big)\psi(x)}\,dx,
\end{align*}
where Green's formula was used for $\Omega'=B_{r}'\subset \R^{d'}$ and $\Omega''=B_{r}''\subset \R^{d''}$ in relevant places.
\end{proof}

\begin{lemma} \label{lem:Sobo}
The norm  generated by the scalar product $\langle \cdot,\cdot\rangle_{G_0}$ in $W^1_{G_0}(\Rd_*)$ is complete.
\end{lemma}
\begin{proof}
Let $\{f_n\}$ be a Cauchy sequence in $W^1_{G_0}(\Rd_*)$. This means that $\{f_n\}$, $\{\partial_jf_n\}$ for $1\le j\le d'$, $\{|x'|\partial_jf_n\}$ 
for $d'+1\le j\le d$, are Cauchy sequences in $L^2(\Rd)$. Let $f_n\to f$, $\partial_jf_n\to g$, and $|x'|\partial_jf_n\to h$ in $L^2(\Rd)$, 
for some $f,g,h\in L^2(\Rd)$ and relevant $j$'s. It follows that for any $\vp\in C^\infty_c(\Rd_*)$, $\langle f_n,\vp\rangle\to \langle f,\vp\rangle$ and 
$\langle \partial_j f_n,\vp\rangle\to \langle g,\vp\rangle$. But 
$$
\langle \partial_j f_n,\vp\rangle=- \langle  f_n,\partial_j\vp\rangle\to - \langle  f,\partial_j\vp\rangle,
$$
hence $\partial_jf=g$ follows. Similarly, $\langle |x'|\partial_j f_n,\vp\rangle\to \langle h,\vp\rangle$ and combining this with
$$
\langle |x'|\partial_j f_n,\vp\rangle=\langle \partial_j f_n,|x'|\vp\rangle=- \langle  f_n,\partial_j(|x'|\vp)\rangle\to -\langle f,\partial_j(|x'|\vp)\rangle,
$$
gives $(|x'|\partial_j)f=h$. This shows that $f\in W^1_{G_0}(\Rd_*)$ and that $\{f_n\}$ converges to $f$ in $W^1_{G_0}(\Rd_*)$.
\end{proof}

Recall that
$$
T_{t,x',\xi''}(y') =\Big(\frac{|\xi''|}{2\pi\sinh 2t|\xi''|}\Big)^{d'/2}\exp\Big(-\frac{|\xi''|}4\big(|x'+y'|^2 \tanh t|\xi''|+ |x'-y'|^2\coth t|\xi''|\big)\Big)
$$
and
$$
p_t(x,y)=\int_{\R^{d''}} e^{\textrm{i}(x''-y'')\cdot\xi''}T_{t,x',\xi''}(y')\,d\xi''.
$$
\begin{lemma} \label{lem:last}
Let  $y=(y',y'')\in\Rd$ be fixed. We have 
\begin{equation}\label{6.4}
\partial_t\int_{\R^{d''}} e^{\textrm{i}(x''-y'')\cdot\xi''}T_{t,x',\xi''}(y')\,d\xi''=\int_{\R^{d''}} e^{\textrm{i}(x''-y'')\cdot\xi''}\partial_tT_{t,x',\xi''}(y')\,d\xi''
\end{equation}
and, for  $j=1,\ldots,d'$, 
\begin{equation}\label{6.5}
\partial_{x_j'}^2\int_{\R^{d''}} e^{\textrm{i}(x''-y'')\cdot\xi''}T_{t,x',\xi''}(y')\,d\xi''=\int_{\R^{d''}} e^{\textrm{i}(x''-y'')\cdot\xi''}\partial_{x_j'}^2T_{t,x',\xi''}(y')\,d\xi''.
\end{equation}
\end{lemma}
\begin{proof}
For \eqref{6.4}, given $x,y\in\Rd$ is fixed, let us  introduce a more convenient notation,
$$
F(t,\xi'')=F_{x',y'}(t,\xi'')=(2\pi)^{d'/2}T_{t,x',\xi''}(y')
$$
and
$$
f(t,\tau)=f_{A,B}(t,\tau)=\Big(\frac{\tau}{\sinh 2t\tau}\Big)^{d'/2}\exp\big(-\tau(A\tanh t\tau+B\coth t\tau)\big), \quad \tau\ge0,
$$
so that, with setting $A:=|x'+y'|^2/4$ and $B:=|x'-y'|^2/4$,  $F(t,\xi'')$ takes the form
$$
F(t,\xi'')=f(t,|\xi''|).
$$
Now, given $t_0\in (0,\infty)$, to check that \eqref{6.4} holds for $t=t_0$ it is sufficient to find  a nonnegative function 
$\Phi(\xi'')=\Phi_{t_0}(\xi'')$ such that $\int_{\R^{d''}}\Phi<\infty$ and $|\partial_tF(t,\xi'')|\le \Phi(\xi'')$ for $t\in(t_0/2,2t_0)$ and $\xi''\in\R^{d''}$. 
But
$$
\partial_t f(t,\tau)=-\tau\Big(\frac{\tau}{\sinh 2t\tau}\Big)^{d'/2}\exp\big(-\tau(A\tanh t\tau+B\coth t\tau)\big)
\Big(d'\coth 2t\tau+ \frac{\tau A}{\cosh^2 t\tau}-\frac{\tau B}{\sinh^2 t\tau}\Big),
$$
and so, neglecting the exponential factor which is $\le1$,
$$
|\partial_t f(t,\tau)|\le \Big(\frac{\tau}{\sinh 2t\tau}\Big)^{d'/2}
\Big(d'\tau\coth 2t\tau + \frac{\tau^2 A}{\cosh^2 t\tau} +\frac{\tau^2 B}{\sinh^2 t\tau}\Big).
$$
Thus, for $t_0>0$ fixed and $t_0/2<t<2t_0$, we have
$$
|\partial_t f(t,\tau)|\le C( \psi_1(t,\tau)+ \psi_2(t,\tau)+\psi_3(t,\tau)),
$$
where $C=C(d',t_0,A,B)>0$ and 
\begin{align*}
\psi_1(t,\tau)&=\Big(\frac{\tau}{\sinh 2t\tau}\Big)^{d'/2+1}\cosh 2t\tau\le \Big(\frac{\tau}{\sinh t_0\tau}\Big)^{d'/2+1}\cosh 4t_0\tau\equiv\phi_1(\tau),\\
\psi_2(t,\tau)&=\Big(\frac{\tau}{\sinh 2t\tau}\Big)^{d'/2}\le \Big(\frac{\tau}{\sinh t_0\tau}\Big)^{d'/2}\equiv\phi_2(\tau),\\
\psi_3(t,\tau)&=\Big(\frac{\tau}{\sinh t\tau}\Big)^{d'/2+2}(\cosh t\tau)^{d'/2}\le \Big(\frac{\tau}{\sinh(t_0\tau/2)}\Big)^{d'/2+2}(\cosh 2t_0\tau)^{d'/2}\equiv\phi_3(\tau).
\end{align*}
It is easily seen that $\Phi(\xi'')=\sum_1^3\phi_j(|\xi''|)$ is integrable on $\R^{d''}$ which finishes the proof of  \eqref{6.4}.

For \eqref{6.5} we first prove its version with $\partial_{x_j'}$ replacing $\partial_{x_j'}^2$. This time we denote 
$$
F(x',\xi'')=F_{x',y'}(t,y')=\exp\Big(-\frac{|\xi''|}4\big(|x'+y'|^2 \tanh t|\xi''|+ |x'-y'|^2\coth t|\xi''|\big)\Big)
$$
and observe that 
\begin{equation}\label{6.6}
\partial_{x_j'}F(x',\xi'')=-\frac{|\xi''|}2 \big((x_j'+y_j')\tanh t|\xi''| + (x_j'-y_j')\coth t|\xi''|\big)F(x',\xi'').
\end{equation}
Thus, for $t\in(0,\infty)$, $y'\in \R^{d'}$ and $x_{j,0}'\in\R$ fixed, neglecting the exponential factor we have for $|x_{j}'-x_{j,0}'|<1$
\begin{align*}
|\partial_{x_j}T_{t,x',\xi''}(y')|&\le \Big(\frac{|\xi''|}{2\pi\sinh 2t|\xi''|}\Big)^{d'/2}|\partial_{x_j} F(x',\xi'') |\\
&\lesssim |\xi''|\Big(\frac{|\xi''|}{\sinh 2t|\xi''|}\Big)^{d'/2}\big(\tanh t|\xi''| + \coth t|\xi''|\big)\\
&\lesssim |\xi''|\Big(\frac{|\xi''|}{\sinh 2t|\xi''|}\Big)^{d'/2}+\Big(\frac{|\xi''|}{\sinh t|\xi''|}\Big)^{d'/2+1}(\cosh  t|\xi''|)^{-d'/2+1},
\end{align*}
where we used $|x_j'\pm y_j'|\le |x_{j,0}'|+|y_j|+1$ for $|x_{j}'-x_{j,0}'|<1$. It is clear that the both functions in the last bound are  integrable on $\R^{d''}$. 

We now pass to \eqref{6.5}. We use \eqref{6.6} to find
$$
\partial_{x_j'}^2F(x',\xi'')=-\frac{|\xi''|}2\Big(\big((x_j'+y_j')\tanh t|\xi''|+(x_j'-y_j')\coth t|\xi''|\big)^2+\tanh t|\xi''|+\coth t|\xi''|\Big)F(x',\xi'').
$$
Similarly to the previous calculations and for $t$, $y'$ and $x_{j,0}'$ as before, it is easy to find $\Phi\in L^1(\R^{d''})$ such that
$$
|\partial_{x_j'}^2T_{t,x',\xi''}(y')|\le \Phi(\xi''),\qquad \xi''\in \R^{d''},\quad |x_{j}'-x_{j,0}'|<1.
$$
This finishes the proof of \eqref{6.5} and thus the lemma.
\end{proof}

We close this section with comment on a possibly alternative definition of $\calG$. Namely, in the computation of $\calG\vp$ for $\vp\in C^\infty_c(\Rd)$ in  Section~\ref{sec:Gru}, we can change the order of integration to end up with $\calG\vp(k,\xi'')=\calF''\big[\calH^{\textsl{sc}}_{\xi''}\vp_{x''}(k)\big](\xi'')$. This suggests an alternative for \eqref{3.3}, namely
\begin{equation*}
\hat{\calG} f(k,\xi''):=\calF'' \big[\calH^{\textsl{sc}}_{\xi''}f_{x''}(k)\big](\xi''),  \qquad f\in L^2(\Rd).
\end{equation*}
With appropriately defined $\hat{\calG}^{-1}$, a mapping from $L^2(\Gamma)$ to the space of functions on $\Rd$, the analogue of Theorem~\ref{thm:first} 
for $\hat{\calG}$ holds true. But $\calG$ and $\hat{\calG}$ coincide on $C^\infty_c(\Rd)$, hence $\calG=\hat{\calG}$, and also $\calG^{-1}=\hat{\calG}^{-1}$. 

\textbf{Acknowledgements}. The author would like to thank Professors Jacek Dziuba\'nski and Alessio Martini for their valuable comments and remarks. 


\end{document}